\documentclass[12pt,leqno]{amsart}
\usepackage{euscript, amssymb, amsmath, amsthm}
\usepackage{epsfig}
\usepackage{graphicx}
\usepackage{caption}
\usepackage{longtable}
\usepackage{dcolumn}
\usepackage{setspace}
\usepackage[most]{tcolorbox}
\definecolor{webred}{rgb}{0.75,0,0}
\definecolor{webgreen}{rgb}{0,0.75,0}
\definecolor{refkey}{gray}{0.75}
\usepackage[pagebackref=true, colorlinks=true, citecolor=blue]{hyperref}

\numberwithin{equation}{section}

\newtheorem{theo}{Theorem}[section]
\newtheorem{lem}{Lemma}[section]

\newtheorem{Def}[theo]{Definition}
\theoremstyle{remark}
\newtheorem{rem}{Remark}[section]

\newcommand{\ep}{\varepsilon}
\def\R{{\mathbb{R}}}

\def\d{\displaystyle}
\def\e{{\varepsilon}}
\def\p{\partial}

\newcommand{\m}{-k}
\newcommand{\M}{-2k}
\newcommand{\D}{\frac{\mu}{t}}

\date{}


\tcbset{
    frame code={}
    center title,
    left=0pt,
    right=0pt,
    top=0pt,
    bottom=0pt,
    colback=gray!10,
    colframe=white,
    width=\dimexpr\textwidth\relax,
    enlarge left by=0mm,
    boxsep=5pt,
    arc=0pt,outer arc=0pt,
    }

\begin{document}

\title[Blow-up and lifespan estimates for a damped wave equation in the Einstein - de Sitter spacetime with nonlinearity of derivative type]{Blow-up and lifespan estimates for a damped wave equation in the Einstein-de Sitter spacetime with nonlinearity of derivative type}
\author[M. Hamouda, M.A. Hamza and A. Palmieri]{Makram Hamouda$^{1}$, Mohamed Ali Hamza$^{1}$  and Alessandro Palmieri$^{2,3}$}
\address{$^{1}$ Department of Basic Sciences, Deanship of Preparatory Year and Supporting Studies, Imam Abdulrahman Bin Faisal University, P.O. Box 1982, 34212 Dammam, Saudi Arabia.}
\address{$^{2}$ Department of Mathematics, University of Pisa, Largo B. Pontecorvo 5, 56127 Pisa, Italy.}
\address{$^{3}$  {\it Current address:} Mathematical Institute, Tohoku University, Aoba, Sendai 980-8578, Japan.}

\medskip

\email{mmhamouda@iau.edu.sa (M. Hamouda)} 
\email{mahamza@iau.edu.sa (M.A. Hamza)}
\email{alessandro.palmieri.math@gmail.com (A. Palmieri)}

\pagestyle{plain}


\begin{abstract}
In this article, we investigate the blow-up for local solutions to a semilinear wave equation in the generalized Einstein - de Sitter spacetime with nonlinearity of derivative type. More precisely, we consider a semilinear damped wave equation with a time-dependent and not summable speed of propagation and with a time-dependent coefficient for the linear damping term with critical decay rate. We prove in this work that the results obtained in a previous work, where the damping coefficient takes two particular values $0$ or $2$, can be extended for any positive damping coefficient. We show the blow-up in finite time of local in time solutions and we establish upper bound estimates for the lifespan, provided that the exponent in the nonlinear term is below a suitable threshold and that the Cauchy data are nonnegative and compactly supported.

\end{abstract}

\maketitle


\section{Introduction}
\par\quad

We are interested in the   semilinear damped wave equation when the   speed of propagation is depending on time, namely the damped wave equations in  Einstein - de Sitter spacetime, with time derivative nonlinearity which reads as follows:
\begin{equation}
\label{T-sys}
\left\{
\begin{array}{l}
\d u_{tt}-t^{\M}\Delta u+\D  u_t=|u_t|^p, 
\quad \mbox{in}\ \R^N\times[1,\infty),\\
u(x,1)=\e f(x),\ u_t(x,1)=\e g(x), \quad  x\in\R^N,
\end{array}
\right.
\end{equation}
where $k\in [0,1)$, $\mu \ge 0$, $p>1$,  $\ N \ge 1$ is the space dimension, $\e>0$ is a parameter illustrating the size of the initial data, 
and  $f,g$ are supposed to be positive functions.
 Furthermore, we consider  $f$ and $g$ with compact support on   $B(0_{\R^N},R), R>0$.

The  problem \eqref{T-sys} with time derivative nonlinearity being replaced by power nonlinearity is well understood in terms of blow-up phenomenon. Let us first recall the equation in this case.  Under the usual Cauchy conditions,  the  semilinear wave equation with power nonlinearity is
\begin{equation}
\label{T-sys-ab}
\d u_{tt}-t^{\M}\Delta u+\D u_t=|u|^q, 
\quad \mbox{in}\ \R^N\times[1,\infty).
\end{equation}
The  blow-up phenomenon for \eqref{T-sys-ab} is  related to two particular exponents. The first exponent, $q_0(N,k)$,   is the positive root of  $$((1-k)N-1)q^2-((1-k)N+1+2k)q-2(1-k)=0,$$
and  the second exponent is given by 
$$q_1(N,k)=1+\frac{2}{N(1-k)}.$$
Hence,  the positive number $\d \max \left(q_0(N+\frac{\mu}{1-k},k),q_1(N,k)\right)$ seems to be a serious candidate for  the critical power stating thus the threshold between the global existence and the blow-up regions, see e.g. \cite{Galstian,Palmieri, Palmieri2,Tsutaya,Tsutaya2}.

Let us go back to \eqref{T-sys} with $k =\mu= 0$. This case  is in fact connected to the Glassey conjecture in which  the critical exponent $p_G$ is given by
\begin{equation}\label{Glassey}
p_G=p_G(N):=1+\frac{2}{N-1}.
\end{equation}
The above  value $p_G$ is creating a threshold (depending on $p$) between the region where we have  the global existence of small data solutions (for $p>p_G$) and another where the blow-up of the solutions under suitable
sign assumptions for the Cauchy data occurs (for $p \le p_G$); see e.g. \cite{Hidano1,Hidano2,John1,Sideris,Tzvetkov,Zhou1}.\\

Now,  for $k<0$ and $\mu=0$, it is proven in \cite{LP-tricomi} that the solution of \eqref{T-sys}, in the subcritical case ($1<p \le p_G(N(1-k))$), blows up in finite time  giving hence a lifespan estimate of the maximal existence time. This is equivalent to say that, for $1<p \le p_G(N(1-k))$, we have the nonexistence of the solution of \eqref{T-sys}. However, the aforementioned result was recently improved in \cite{Lai2020} thanks to the construction of adequate test functions. The new  region obtained in  \cite{Lai2020} gives a plausible characterization of the critical exponent, namely  
\begin{equation}\label{pT}
p \le p_T(N,k):=1+\frac{2}{(1-k)(N-1)+k}.
\end{equation}
Very recently, it is proved in \cite{Our5} with different approaches, 
as an application of the case of 
mixed nonlinearities, that results similar to the above for the problem 
\eqref{T-sys} with $k<0$ and $\mu=0$ hold. \\

We consider now the case $\mu > 0$ and $k=0$ in \eqref{T-sys-ab}. Hence, for a small  $\mu$,  the solution of \eqref{T-sys-ab} behaves  like a wave. In fact, the damping produces  a shifting by $\mu>0$ on the dimension $N$ for the value of the critical power, 
see e.g. \cite{Ikeda, Palmieri-Reissig,  Tu-Lin1, Tu-Lin}, and \cite{Dab1,Dab2} for the case 
 $\mu=2$ and $N=2,3$. The global existence for $\mu=2$ is proven in \cite{Dab1, Dab2, Palmieri2-2019}. 
 However, for $\mu$  large,  the equation \eqref{T-sys-ab} is of a parabolic type and the behavior is like a heat-type equation; see e.g. \cite{dabbicco1,dabbicco2,wakasugi}.

On the other hand, for the solution of \eqref{T-sys} with  $\mu > 0$ and $k=0$,  in \cite{LT2} a blow-up result is proved for $1<p\leq p_G(N+2\mu)$ and upper bound estimates for the lifespan are given as well. Later,  this result was improved in \cite{Palmieri-Tu-2019}, where $p_G(N+ \mu)$ is found as upper bound for $\mu \ge 2$.  Recently,  an improvement is obtained in \cite{Our2} stating that the critical value for $p$ is given by $p_G(N+ \mu)$ for all $\mu >0$. This should be the optimal threshold value that needs to be rigorously proved by completing the present blow-up result with a global existence one when the exponent $p$ is beyond the critical value.  \\

We focus in this article on the blow-up  of the solution of (\ref{T-sys}) for $k\in [0,1)$.   Our target is to give the upper bound, denoted here by $p_{E}=p_{E}(N,k,\mu)$, delimiting a new blow-up region for the Einstein - de Sitter spacetime equation   \eqref{T-sys}. 

First, as observed for the equation \eqref{T-sys-ab},  where the damping produces a shift in $q_0$ in the dimensional parameter of magnitude $\d \frac{\mu}{1-k}$,  we expect that the same phenomenon holds for \eqref{T-sys}.  In other words, we predict that the upper bound $p_{E}=p_{E}(N,k,\mu)$ satisfies
 \begin{equation}\label{pEpE}
 p_{E}(N,k,\mu)=p_E(N+\frac{\mu}{1-k},k,0).
 \end{equation}
 
Using an explicit representation formula and Zhou's approach to proving the blow-up on a certain characteristic line, in \cite{HHP1}, we proved that
\begin{equation}\label{pEpT}
p_{E}(N,k,0)=p_T(N,k),
  \end{equation}
  where $p_T$ is defined by \eqref{pT}.\\
 Now, in view of \eqref{pEpE} and \eqref{pEpT}, we await, for the solution of  \eqref{T-sys} with $k \in [0,1)$ and $\mu>0$, that
  \begin{equation}\label{critical-exp}
 p_{E}=p_{E}(N,k,\mu):=1+\frac{2}{(1-k)(N-1)+k+\mu}.
 \end{equation}

As we have mentioned, in  \cite{HHP1} we proved that \eqref{critical-exp} holds true under some sign assumptions for the data for $\mu=0$, but also for $\mu=2$ (cf. Theorems 1.1 and 1.2). We aim in the present work to extend this result for all $\mu>0$, and show that the upper bound value for $p$ is in fact given by \eqref{critical-exp}. We think that $p_{E}(N,k,\mu)$, for $k$ small, characterizes the limiting value between the existence and nonexistence regions of the solution of \eqref{T-sys}. However, it is clear that this limiting exponent does not reach the optimal one in view of the very recent results in \cite{Tsutaya3}.


Finally, we recall here that the wave in (\ref{T-sys}) has a  speed of propagation dependent of time. Therefore,   this time-dependent speed of propagation term can be seen, after rescaling (see \eqref{eq-v} below), as a scale-invariant damping. Let $v(x,\tau)=u(x,t)$, where 
\begin{equation}\label{xi}
\tau=\phi_k(t):=\frac{t^{1-k}}{1-k}.
\end{equation} 
Hence, we can easily see that $v(x,\tau)$ satisfies the following equation:
\begin{align} \label{eq-v}
v_{\tau\tau}-\Delta v+ \frac{\mu-k}{(1-k)\tau} \,\partial_\tau v = C_{k,p} \tau^{\mu_k (p-2)} |\partial_{\tau} v|^p, 
\quad \mbox{in}\ \R^N\times[1/(1-k),\infty),
\end{align} 
where $\mu_k := \frac{\m}{1-k}$ and $C_{k,p}= (1-k)^{\mu_k(p-2)}$.  Moreover, thanks to the above transformation, we can use the methods carried out in some earlier works \cite{CLP,Our,Our2, Our3, Our5} to build the proof of our main result.
\\

The rest of the paper is arranged as follows. First, we state in  Section \ref{sec-main}  the weak formulation of (\ref{T-sys}) in the energy space, and then we give the main theorem.  Section \ref{aux} is concerned with  some technical lemmas that we will use to prove the main result. Finally, Section \ref{sec-ut} is assigned  to the proof of Theorem \ref{th_u_t} which constitutes the main result of this article.

\section{Nonexistence Result}\label{sec-main}
\par

First, we define in the sequel the energy solution associated with (\ref{T-sys}).
\begin{Def}\label{def1}
 Let $f\in H^{1}(\R^N)$ and $g \in L^{2}(\R^N)$. The function $u$ is said to be an energy  solution of
 (\ref{T-sys}) on $[1,T)$
if
\begin{displaymath}
\left\{
\begin{array}{l}
u\in \mathcal{C}([1,T),H^1(\R^N))\cap \mathcal{C}^1([1,T),L^2(\R^N)), \vspace{.1cm}\\
 \ u_t \in L^p_{loc}((1,T)\times \R^N),
 \end{array}
  \right.
\end{displaymath}
satisfies, for all $\Phi\in \mathcal{C}_0^{\infty}(\R^N\times[1,T))$ and all $t\in[1,T)$, the following equation:
\begin{equation}
\label{energysol2-21}
\begin{array}{l}
\d\int_{\R^N}u_t(x,t)\Phi(x,t)dx-\e \int_{\R^N}g(x)\Phi(x,1)dx \vspace{.2cm}\\
\d -\int_1^t  \int_{\R^N}u_t(x,s)\Phi_t(x,s)dx \,ds+\int_1^t s^{\M} \int_{\R^N}\nabla u(x,s)\cdot\nabla\Phi(x,s) dx \,ds\vspace{.2cm}\\
\d  +\int_1^t  \int_{\R^N}\frac{\mu}{s}u_t(x,s) \Phi(x,s)dx \,ds=\int_1^t \int_{\R^N}|u_t(x,s)|^p\Phi(x,s)dx \,ds,
\end{array}
\end{equation}
 and the condition $u(x,1)=\varepsilon f(x)$ is fulfilled in $H^1(\mathbb{R}^N)$.\\
A straightforward computation shows that \eqref{energysol2-21} is equivalent to
\begin{equation}
\begin{array}{l}\label{energysol2}
\d \int_{\R^N}\big[u_t(x,t)\Phi(x,t)- u(x,t)\Phi_t(x,t)+\frac{\mu}{t}u(x,t) \Phi(x,t)\big] dx \vspace{.2cm}\\
\d \int_1^t  \int_{\R^N}u(x,s)\left[\Phi_{tt}(x,s)-s^{-2k}\Delta \Phi(x,s) -\frac{\partial}{\partial s}\left(\frac{\mu}{s}\Phi(x,s)\right)\right]dx \,ds\vspace{.2cm}\\
\d =\int_{1}^{t}\int_{\R^N}|u_t(x,s)|^p\psi(x,s)dx \, ds + \e \int_{\R^N}\big[-f(x)\Phi_t(x,1)+\left(\mu f(x)+g(x)\right)\Phi(x,1)\big]dx.
\end{array}
\end{equation}
\end{Def}

\begin{rem}\label{rem-supp}
Obviously, we can choose a test function $\Phi$  which is not compactly supported in view of the fact that the initial data $f$ and $g$ are supported on $B_{\R^N}(0,R)$. In fact, we have $\mbox{\rm supp}(u)\ \subset\{(x,t)\in\R^N\times[1,\infty): |x|\le \phi_k(t)+R\}$.
\end{rem}

The blow-up region and the lifespan estimate of the solutions of  (\ref{T-sys}) constitute the objective of our main result which is the subject of the following theorem.
\begin{theo}
\label{th_u_t}
Let $\mu >0$, $p \in (1, p_{E}(N,k,\mu)], N \ge 1$ and $k \in [0,1)$. Suppose that $f\in H^{1}(\R^N)$ and $g \in L^{2}(\R^N)$ are  functions which are non-negative, with compact support on  $B(0_{\R^N},R)$,
and  non-vanishing everywhere. Then,  there exists $\e_0=\e_0(f,g,N,R,p,k,\mu)>0$ such that for any $0<\e\le\e_0$ the solution $u$  to \eqref{T-sys} which satisfies 
$$\mbox{\rm supp}(u)\ \subset\{(x,t)\in\R^N\times[1,\infty): |x|\le \phi_k(t)+R\},$$
blows up in finite time $T_\e$, and 
\begin{displaymath}
T_\e \leq
\d \left\{
\begin{array}{ll}
 C \, \e^{-\frac{2(p-1)}{2-((1-k)(N-1)+k+\mu)(p-1)}}
 &
 \ \text{for} \
 1<p<p_{E}(N,k,\mu), \vspace{.1cm}
 \\
 \exp\left(C\e^{-(p-1)}\right)
&
 \ \text{for} \ p=p_{E}(N,k,\mu),
\end{array}
\right.
\end{displaymath}
 where $p_{E}(N,k,\mu)$ is given by \eqref{critical-exp} and $C$ is a positive constant independent of $\e$.
\end{theo}

\begin{rem}
The results stated in Theorem \ref{th_u_t} hold true for $k<0$ and $\mu>0$; see \cite{Our6} where a more general model with mass term is studied. 
\end{rem}

\begin{rem} After completing the first version of the present manuscript, we received a draft version of \cite{Tsutaya3}, where problem \eqref{T-sys} is studied, among other things. In particular, for $\frac{n+1}{n+2}<k<1$ and $\mu\in [0,(n+2)k-(n+1))$ the upper bound for $p$ in the blow-up result is improved in \cite{Tsutaya3} by proving the nonexistence of global solutions to \eqref{T-sys} for $1<p<1+\frac{1}{(1-k)n+\mu}$.
\end{rem}

\section{Auxiliary results}\label{aux}
\par

It is worth mentioning that the choice of the test function, that we will use in the functionals that will be  introduced later  on, is crucial here. Naturally, in terms of dynamics of the solution of \eqref{T-sys},  the more accurate the choice of the test function is, the better lifespan estimate we obtain. This is why we choose in the following to include all the linear terms inherited from \eqref{T-sys}. First, we introduce  the function $\rho(t)$ \cite{Palmieri}  given by
\begin{equation}\label{lmabdaK}
\rho(t):=\d t^{\frac{1+\mu}{2}}K_{\frac{\mu-1}{2(1-k)}}\left(\frac{t^{1-k}}{1-k}\right), \quad \forall \ t\ge 1,
\end{equation}
where $K_{\nu}(t)$ is the modified Bessel function of  second kind defined as
\begin{equation}\label{bessel}
K_{\nu}(t)=\int_0^\infty\exp(-t\cosh \zeta)\cosh(\nu \zeta)d\zeta,\ \nu\in \mathbb{R}.
\end{equation}
It is easy to see that $\rho(t)$ satisfies
\begin{equation}\label{lambda}
\d \frac{d^2 \rho(t)}{dt^2}-t^{\M}\rho(t)-\frac{d}{dt}\left(\frac{\mu}{t}\rho(t)\right)=0, \quad \forall \ t\ge 1.
\end{equation}
Second, we define the function $\varphi(x)$ by
\begin{equation}
\label{test11}
\varphi(x):=
\left\{
\begin{array}{ll}
\d\int_{S^{N-1}}e^{x\cdot\omega}d\omega & \mbox{for}\ N\ge2,\vspace{.2cm}\\
e^x+e^{-x} & \mbox{for}\  N=1;
\end{array}
\right.
\end{equation}
note that $\varphi(x)$ is introduced in \cite{YZ06} and satisfies $\Delta\varphi=\varphi$.\\
Hence,  the function $\psi(x,t):=\varphi(x)  \rho(t)$ verifies the following equation:
\begin{equation}\label{lambda-eq}
\partial^2_t \psi(x, t)-t^{\M}\Delta \psi(x, t) -\frac{\p}{\p t}\left(\frac{\mu}{t}\psi(x, t)\right)=0.
\end{equation}

In the following we enumerate some properties of the function $\rho(t)$ that we will use later on in the proof of our main result. 

\begin{lem}\label{lem-supp}
The next properties hold true for the function $\rho(t)$.
\begin{itemize}
\item[{\bf (i)}] The function $\rho(t)$ is positive  on $[1,\infty)$. Moreover,  for all $t \ge 1$, there exists a constant $C_1$ such that   $\rho(t)$ satisfies 
\begin{equation}\label{est-rho}
C_1^{-1}t^{\frac{k+\mu}{2}} \exp(-\phi_k(t)) \le  \rho(t) \le C_1 t^{\frac{k+\mu}{2}} \exp(-\phi_k(t)), 
\end{equation}
where $\phi_k(t)$ is given by \eqref{xi}.
\item[{\bf (ii)}] We have
\begin{equation}\label{lambda'lambda1}
\d \lim_{t \to +\infty} \left(\frac{t^k \rho'(t)}{\rho(t)}\right)=-1.
\end{equation}
\end{itemize}
\end{lem}
\begin{proof}
First, we recall here the definition of  $\rho(t)$, as in \eqref{lmabdaK}, and \eqref{xi}
\begin{equation}\label{lmabdaK1}
\rho(t)=\d t^{\frac{1+\mu}{2}}K_{\frac{\mu-1}{2(1-k)}}\left(\phi_k(t)\right), \quad \forall \ t\ge 1.
\end{equation}
Hence, the positivity of $\rho(t)$ is straightforward thanks to \eqref{bessel}.
On the other hand, from \cite{Gaunt}, we have the following property for the function $K_{\mu}(t)$:
\begin{equation}\label{Kmu}
K_{\mu}(t)=\sqrt{\frac{\pi}{2t}}e^{-t} (1+O(t^{-1})), \quad \text{as} \ t \to \infty.
\end{equation}
Combining \eqref{lmabdaK1} and \eqref{Kmu}, and again remembering the definition of $\phi_k(t)$, given by \eqref{xi}, and the fact that $k <1$, we conclude \eqref{est-rho}. The  assertion {\bf (i)} is thus proven.

Now, to prove {\bf (ii)}, using \eqref{lmabdaK1} we observe that 
\begin{equation}\label{lambda'lambda}
\d \frac{\rho'(t)}{\rho(t)}=\frac{\mu+1}{2t}+t^{-k}\frac{K'_{\frac{\mu-1}{2(1-k)}}\left(\phi_k(t)\right)}{K_{\frac{\mu-1}{2(1-k)}}\left(\phi_k(t)\right)}.
\end{equation}
Exploiting the well-known identity for the modified Bessel function, 
\begin{equation}\label{Knu-pp}
\frac{d}{dz}K_{\nu}(z)=-K_{\nu+1}(z)+\frac{\nu}{z}K_{\nu}(z),
\end{equation}
and combining \eqref{lambda'lambda} and \eqref{Knu-pp} yields
\begin{equation}\label{lambda'lambda2}
\d \frac{\rho'(t)}{\rho(t)}=\frac{\mu}{t}-t^{-k}\frac{K_{1+\frac{\mu-1}{2(1-k)}}\left(\phi_k(t)\right)}{K_{\frac{\mu-1}{2(1-k)}}\left(\phi_k(t)\right)}.
\end{equation}
From \eqref{Kmu} and \eqref{lambda'lambda2}, and using the fact that $k \in [0,1)$, we deduce \eqref{lambda'lambda1}.

This ends the proof of Lemma \ref{lem-supp}.
\end{proof}


Throughout this article, the use of a generic parameter $C$ is designed to denote a positive constant that might be dependent on $p,q,k,N,R,f,g, \mu$ but independent of $\ep$. The  value of the constant $C$ may change from line to line. Nevertheless,  when it is necessary, we will clearly mention the expression of $C$ in terms of the parameters involved in our problem.\\

A classical estimate result for the function $\psi(x, t)$ is stated in the next lemma.
\begin{lem}[\cite{YZ06}]
\label{lem1} Let  $r>1$. Then, there exists a constant $C=C(N,\mu,R,p,k,r)>0$ such that
\begin{equation}
\label{psi}
\int_{|x|\leq \phi_k(t)+R}\Big(\psi(x,t)\Big)^{r}dx
\leq C\rho^r(t)e^{r\phi_k(t)}(1+\phi_k(t))^{\frac{(2-r)(N-1)}{2}},
\quad\forall \ t\ge 1.
\end{equation}
\end{lem}
\par
Let $u$ be a solution to \eqref{T-sys} for which we introduce the following functionals:
\begin{equation}
\label{F1def}
\mathcal{U}(t):=\int_{\R^N}u(x, t)\psi(x, t)dx,
\end{equation}
and
\begin{equation}
\label{F2def}
\mathcal{V}(t):=\int_{\R^N}u_t(x,t)\psi(x, t)dx.
\end{equation}
The first  lower bounds for $\mathcal{U}(t)$ and $\mathcal{V}(t)$ are respectively given by the following two lemmas where, for $t$ large enough, we will prove that $\e^{-1} t^{-k}\mathcal{U}(t)$ and $\e^{-1}\mathcal{V}(t)$ are two bounded from below functions by  positive constants. 
\begin{lem}
\label{F1}
Let $u$ be a solution of  \eqref{T-sys}. Assume in addition that the corresponding initial data satisfy the assumptions as in Theorem \ref{th_u_t}. Then, there exists $T_0=T_0(k,\mu)>2$ such that 
\begin{equation}
\label{F1postive}
\mathcal{U}(t)\ge C_{\mathcal{U}}\, \e \, t^{k}, 
\quad\text{for all}\ t \ge T_0,
\end{equation}
where $C_{\mathcal{U}}$ is a positive constant that may depend on $f$, $g$, $N,\mu,R$ and $k$, but not on $\e$.
\end{lem}
\begin{proof} 
Let $ t \in (1,T)$.  Substituting in \eqref{energysol2}  $\Phi(x, t)$ by $\psi(x, t)$,
 we obtain
\begin{equation}
\begin{array}{l}\label{eq5-12}
\d \int_{\R^N}\big[u_t(x,t)\psi(x,t)- u(x,t)\psi_t(x,t)+\frac{\mu}{t}u(x,t) \psi(x,t)\big] dx \\
\hspace{2cm}\d =\int_{1}^{t}\int_{\R^N}|u_t(x,s)|^p\psi(x,s)dx \, ds + \e C(f,g),
\end{array}
\end{equation}
where 
\begin{equation}\label{Cfg}
C(f,g):=\rho(1)\int_{\R^N}\big[\big(\mu-\frac{\rho'(1)}{\rho(1)}\big)f(x)+g(x)\big]\phi(x)dx.
\end{equation}
Note that $C(f,g)$ is positive thanks to the fact that $\rho(1)$ and $\mu-\frac{\rho'(1)}{\rho(1)}$ are positive as well (in view of \eqref{lambda'lambda2})  and the sign of the initial data. 
Hence, recall the definition of $\mathcal{U}$, as in \eqref{F1def},  and \eqref{test11},  \eqref{eq5-12} gives
\begin{equation}
\begin{array}{l}\label{eq6}
\d \mathcal{U}'(t)+\Gamma(t)\mathcal{U}(t)=\int_{1}^{t}\int_{\R^N}|u_t(x,s)|^p\psi(x,s)dx \, ds +\e \, C(f,g),
\end{array}
\end{equation}
where 
\begin{equation}\label{gamma}
\Gamma(t):=\frac{\mu}{t}-2\frac{\rho'(t)}{\rho(t)}.
\end{equation}
Neglecting the nonlinear term in \eqref{eq6}, then multiplying the resulting equation from \eqref{eq6} by $\frac{t^{\mu}}{\rho^2(t)}$ and integrating on $(1,t)$, we get
\begin{align}\label{est-G1}
 \mathcal{U}(t)
\ge \mathcal{U}(1)\frac{\rho^2(t)}{t^{\mu}\rho^2(1)}+{\e}C(f,g)\frac{\rho^2(t)}{t^{\mu}}\int_1^t\frac{s^{\mu}}{\rho^2(s)}ds.
\end{align}
From \eqref{lmabdaK}, the definition of  $\phi_k(t)$, given by \eqref{xi}, and using the fact that $\mathcal{U}(1)>0$, the estimate \eqref{est-G1} implies that
\begin{align}\label{est-G1-1}
 \mathcal{U}(t)
\ge {\e}C(f,g) t K^2_{\frac{\mu-1}{2(1-k)}}\left(\phi_k(t)\right)\int^t_{t/2}\frac{1}{sK^2_{\frac{\mu-1}{2(1-k)}}\left(\phi_k(s)\right)}ds, \quad \forall \ t \ge 2.
\end{align}
In view of \eqref{Kmu}, we deduce the existence of $T_0=T_0(k,\mu)>2$ such that 
\begin{align}\label{est-double}
\phi_k(t)K^2_{\frac{\mu-1}{2(1-k)}}(\phi_k(t))>\frac{\pi}{4} e^{-2\phi_k(t)} \quad \text{and}  \quad \phi_k(t)^{-1}K^{-2}_{\frac{\mu-1}{2(1-k)}}(\phi_k(t))>\frac{1}{\pi} e^{2\phi_k(t)}, \ \forall \ t \ge T_0/2.
\end{align}
Inserting \eqref{est-double} in  \eqref{est-G1-1} and using \eqref{xi},   we obtain that
\begin{align}\label{est-U-2}
 \mathcal{U}(t)
&\ge \e \frac{C(f,g)}{4}t^{k}e^{-2\phi_k(t)}\int^t_{t/2}\phi_k'(s)e^{2\phi_k(s)}ds\\
&\ge  \e \frac{C(f,g)}{8}t^{k}[1-e^{-2(\phi_k(t)-\phi_k(t/2))}], \ \forall \ t \ge T_0.\nonumber
\end{align}
Thanks to \eqref{xi} and the fact that $k <1$, we observe that $t \mapsto 1-e^{-2(\phi_k(t)-\phi_k(t/2))}$ is an increasing function on $(T_0, \infty)$, hence, its minimum is achieved at $t=T_0$. Therefore  we deduce that
\begin{align}\label{est-G1-3}
 \mathcal{U}(t)
\ge \e \kappa C(f,g)t^{k}, \ \forall \ t \ge T_0,
\end{align}
where $$\d \kappa:=\frac{1}{8} \left(1-\exp \left(-\frac{(2-2^{k})T_0^{1-k}}{1-k}\right)\right).$$

Hence, Lemma \ref{F1} is now proved.
\end{proof}

The next lemma gives the lower bound of the functional $\mathcal{V}(t)$.
\begin{lem}\label{F11}
Assume that  the initial data are as in Theorem \ref{th_u_t}. For $u$  an energy solution of  \eqref{T-sys},  there exists $T_1=T_1(k,\mu)>T_0$ such that 
\begin{equation}
\label{F2postive}
\mathcal{V}(t)\ge C_{\mathcal{V}}\, \e, 
\quad\text{for all}\ t  \ge  T_1,
\end{equation}
where $C_{\mathcal{V}}$ is a positive constant depending on $f$, $g$, $N,\mu,R$ and $k$,  but not on $\e$.
\end{lem}
 
\begin{proof}
Let $t \in [1,T)$. Recall the definitions of $\mathcal{U}$ and  $\mathcal{V}$, given respectively by \eqref{F1def} and  \eqref{F2def}, \eqref{test11} and the identity
 \begin{equation}\label{def23}\d \mathcal{U}'(t) -\frac{\rho'(t)}{\rho(t)}\mathcal{U}(t)= \mathcal{V}(t).\end{equation}
Hence, the equation  \eqref{eq6} yields
\begin{equation}
\begin{array}{l}\label{eq5bis}
\d \mathcal{V}(t)+\left[\frac{\mu}{t}-\frac{\rho'(t)}{\rho(t)}\right]\mathcal{U}(t)
=\d \int_{1}^{t}\int_{\R^N}|u_t(x,s)|^p\psi(x,s)dx \, ds +\e \, C(f,g).
\end{array}
\end{equation}
A differentiation  in time of the  equation \eqref{eq5bis} gives
\begin{align}\label{F1+bis}
\d \mathcal{V}'(t)+\left[\frac{\mu}{t}-\frac{\rho'(t)}{\rho(t)}\right]\mathcal{U}'(t)-\left(\frac{\mu}{t^2}+\frac{\rho''(t)\rho(t)-(\rho'(t))^2}{\rho^2(t)}\right)\mathcal{U}(t) 
\d =\int_{\R^N}|u_t(x,t)|^p\psi(x,t)dx.
\end{align}
Now, thanks to  \eqref{lambda} and   \eqref{def23}, we deduce from  \eqref{F1+bis} that
\begin{align}\label{F1+bis2}
\d \mathcal{V}'(t)+\left[\frac{\mu}{t}-\frac{\rho'(t)}{\rho(t)}\right]\mathcal{V}(t) =  t^{\M}\mathcal{U}(t)+\int_{\R^N}|u_t(x,t)|^p\psi(x,t)dx,
\end{align}
that we rewrite as
\begin{align}\label{G2++}
\d \left(t^{\mu}\frac{\mathcal{V}(t)}{\rho(t)}\right)' =\frac{t^{\mu}}{\rho(t)} \left( t^{\M}\mathcal{U}(t)+\int_{\R^N}|u_t(x,t)|^p\psi(x,t)dx\right), \quad \forall \ t \ge 1.
\end{align}
An integration of \eqref{G2++} over $(1,t)$ implies that
\begin{align}\label{G23+}
\d t^{\mu}\frac{\mathcal{V}(t)}{\rho(t)} = \frac{\mathcal{V}(1)}{\rho(1)}+\int_1^t\frac{s^{\mu}}{\rho(s)} \left( s^{\M}\mathcal{U}(s)+\int_{\R^N}|u_t(x,s)|^p\psi(x,s)dx\right) ds, \quad \forall \ t \ge 1.
\end{align}
Thanks to the fact that $\mathcal{V}(1) \ge 0$, $\rho(1) >0$ and using the lower bound of $\mathcal{U}$ as  in \eqref{F1postive}, we infer that 
\begin{align}\label{G24+}
\d \mathcal{V}(t) \ge \frac{\rho(t)}{t^{\mu}}\int_1^t\frac{s^{\mu}}{\rho(s)} \left( C_{\mathcal{U}} \ep s^{-k}+\int_{\R^N}|u_t(x,s)|^p\psi(x,s)dx\right) ds, \quad \forall \ t \ge T_0.
\end{align}
Therefore the estimate \eqref{G24+} gives
\begin{align}\label{est-G111-3}
 \mathcal{V}(t)
\ge C_{\mathcal{U}}\,{\e}\frac{\rho(t)}{t^{\mu}}\int_{t/2}^t\frac{s^{\m+\mu}}{\rho(s)}ds, \quad \forall \ t  \ge  2T_0.
\end{align}
For convenience, we rewrite \eqref{est-double} as follows:
\begin{align}\label{est-double-1}
\sqrt{\phi_k(t)}K_{\frac{\mu-1}{2(1-k)}}(\phi_k(t))>\frac{\sqrt{\pi}}{2} e^{-\phi_k(t)} \quad \text{and}  \quad \frac{1}{\sqrt{\phi_k(t)}} K^{-1}_{\frac{\mu-1}{2(1-k)}}(\phi_k(t))>\frac{1}{\sqrt{\pi}} e^{\phi_k(t)}, \ \forall \ t \ge T_0/2.
\end{align}
Using  the expressions of $\rho(t)$ and $\phi_k(t)$, given respectively by \eqref{lmabdaK} and \eqref{xi}, we deduce that
\begin{align}\label{est-G1-2}
 \mathcal{V}(t)
&\d \ge \e \, C_{\mathcal{U}}\left(\frac{1}{2}\right)^{\frac{\mu}{2}+1}e^{-\phi_k(t)}\int^t_{t/2}\phi_k'(s)e^{\phi_k(s)}ds\\
&\d \ge \e \, C_{\mathcal{U}}\left(\frac{1}{2}\right)^{\frac{\mu}{2}+1} [1-e^{-(\phi_k(t)-\phi_k(t/2))}], \ \forall \ t \ge 2T_0.\nonumber
\end{align}
Analogously as in Lemma \ref{F1},  we have
\begin{align}\label{est-G1-2 bis}
 \mathcal{V}(t)
\ge  C_{\mathcal{V}}\,{\e}, \quad \forall \ t \ge T_1:= 2T_0,
\end{align}
where
$$\d C_{\mathcal{V}}:= C_{\mathcal{U}}\left(\frac{1}{2}\right)^{\frac{\mu}{2}+1}\left(1-\exp \left(-\frac{(1-2^{k-1})(2T_0)^{1-k}}{1-k}\right)\right).$$

 This completes the proof of Lemma
\ref{F11}.
\end{proof}

\section{Proof of Theorem \ref{th_u_t}.}\label{sec-ut}

This section is dedicated to proving the main result in  Theorem \ref{th_u_t} which exposes the blow-up dynamics of the solution of \eqref{T-sys}. Hence,  to prove the blow-up result for \eqref{T-sys}  we will use  \eqref{eq5bis} and \eqref{F1+bis2}. For this purpose, we multiply   \eqref{eq5bis} by $\alpha \frac{\rho'(t)}{\rho(t)}$, and subtract the resulting equation from   \eqref{F1+bis2}. Therefore we obtain for a certain $\alpha\ge0$, whose range will be fixed afterward, 
\begin{equation}\label{G2+bis55}
\begin{array}{c}
\d \mathcal{V}'(t)+\left[\frac{\mu}{t}-(1+\alpha) \frac{\rho'(t)}{\rho(t)}\right] \mathcal{V}(t)= -\e \alpha \frac{\rho'(t)}{\rho(t)}\  C(f,g) 
+\d 
\left[t^{-2k} +
\alpha \frac{\rho'(t)}{\rho(t)}\
\d  \left(\frac{\mu}{t}-\frac{\rho'(t)}{\rho(t)}\right)\right]\mathcal{U}(t) \vspace{.2cm}\\
\d +
\int_{\R^N}|u_t(x,t)|^p\psi(x,t)dx -\d \alpha \frac{\rho'(t)}{\rho(t)}\ \d \int_1^t \int_{\R^N}|u_t(x,s)|^p\psi(x,s)dx  ds, \quad \forall \ t  \ge  1.
\end{array}
\end{equation}
Using   \eqref{lambda'lambda1}, we can choose   $\tilde{T}_2\ge T_1$ ($T_1$ is given in Lemma \ref{F11}) such that 
\begin{equation}\label{G2+bis555}
\begin{array}{c}
\d \mathcal{V}'(t)+\left[\frac{\mu}{t}-(1+\alpha) \frac{\rho'(t)}{\rho(t)}\right] \mathcal{V}(t)\ge \frac{\e \alpha t^{-k}}2 \  C(f,g) 
+\d (1-4\alpha)t^{-2k} 
\mathcal{U}(t) \vspace{.2cm}\\
\d +\int_{\R^N}|u_t(x,t)|^p\psi(x,t)dx +\d \frac{\alpha t^{-k}}2  \d \int_1^t \int_{\R^N}|u_t(x,s)|^p\psi(x,s)dx  ds, \quad \forall \ t  \ge  \tilde{T}_2.
\end{array}
\end{equation}
From now on the parameter $\alpha$ is chosen in $(1/7,1/4)$. Thanks to \eqref{F1postive}, the 
 estimate \eqref{G2+bis555} leads to the following lower bound: 
\begin{equation}\label{G2+bis555bb}
\begin{array}{c}
\d \mathcal{V}'(t)+\left[\frac{\mu}{t}-(1+\alpha) \frac{\rho'(t)}{\rho(t)}\right] \mathcal{V}(t)\ge \frac{\e \alpha t^{-k}}2 \  C(f,g) 
+\int_{\R^N}|u_t(x,t)|^p\psi(x,t)dx \vspace{.2cm}\\
\d +\frac{\alpha t^{-k}}2  \d \int_1^t \int_{\R^N}|u_t(x,s)|^p\psi(x,s)dx  ds, \quad \forall \ t  \ge  \tilde{T}_2.
\end{array}
\end{equation}
Now, we introduce the following functional:
\[
H(t):=
C_2 \, \e+\frac1{16}
\int_{\tilde{T}_3}^t \int_{\R^N}|u_t(x,s)|^p\psi(x,s)dx
 ds,
\]
where $C_2:=\min(\alpha C(f,g)/4(1+\alpha) ,C_{\mathcal{V}})$ ($C_{\mathcal{V}}$ is given by Lemma \ref{F11}) and we choose
$\tilde{T}_3>\tilde{T}_2$  such that
\begin{equation}\label{Cond-T3}
\frac{ \alpha }2 \  C(f,g) 
-C_2 t^k\left(\frac{\mu}{t}-(1+\alpha) \frac{\rho'(t)}{\rho(t)}\right)\ge 0,
\end{equation}
and
\begin{equation}\label{Cond-T3-bis}
\frac{\alpha}{2}-
\frac1{16} t^k\left(\frac{\mu}{t}-(1+\alpha) \frac{\rho'(t)}{\rho(t)}\right)
 \ge 0,
 \end{equation}
 for all $t \ge\tilde{T}_3$ (this is possible thanks to \eqref{lambda'lambda1}, the definition of $C_2$  and the fact that $\alpha \in (1/7,1/4)$). \\
Let
$$\mathcal{F}(t):=\mathcal{V}(t)-H(t),$$
which satisfies
\begin{equation}\label{G2+bis6}
\begin{array}{l}
\d \mathcal{F}'(t)+\left[\frac{\mu}{t}-(1+\alpha) \frac{\rho'(t)}{\rho(t)}\right] \mathcal{F}(t)\ge \frac{15}{16} \int_{\R^N}|u_t(x,t)|^p\psi(x,t)dx \vspace{.2cm}\\
\hspace{1cm}+ \d \left[\frac{\alpha}{2}-
\frac1{16} \left(\frac{\mu}{t^{1-k}}-(1+\alpha) \frac{t^k\rho'(t)}{\rho(t)}\right)
\right]  t^{\m}\int_{\tilde{T}_3}^t  \int_{\R^N}|u_t(x,s)|^p\psi(x,s)dx ds\vspace{.2cm}\\ 
\hspace{1cm}+\d \left[ \frac{ \alpha }2 \  C(f,g) 
-C_2\left(\frac{\mu}{t^{1-k}}-(1+\alpha) \frac{t^k\rho'(t)}{\rho(t)}\right)
\right] \e t^{-k},  \quad \forall \ t \ge \tilde{T}_3.
\end{array}
\end{equation}
Thanks to \eqref{Cond-T3} and \eqref{Cond-T3-bis}, we easily conclude that
\begin{equation}\label{G2+bis7}
\d \mathcal{F}'(t)+\left[\frac{\mu}{t}-(1+\alpha) \frac{\rho'(t)}{\rho(t)}\right] \mathcal{F}(t) \ge 0,  \quad \forall \ t \ge \tilde{T}_3.
\end{equation}
Multiplying  \eqref{G2+bis7} by $\frac{t^{ \mu}}{\rho^{1+\alpha}(t)}$ and integrating over $(\tilde{T}_3,t)$, we get
\begin{align}\label{est-G111}
 \mathcal{F}(t)
\ge \mathcal{F}(\tilde{T}_3)\frac{\tilde{T}_3^{\mu}\rho^{1+\alpha}(t)}{t^{ \mu}\rho^{1+\alpha}(\tilde{T}_3)}, \quad \forall \ t \ge \tilde{T}_3.
\end{align}
Hence, we see that $\d \mathcal{F}(\tilde{T}_3)=\mathcal{V}(\tilde{T}_3)-C_2 \e \ge \mathcal{V}(\tilde{T}_3)-C_{\mathcal{V}}\e \ge 0$ in view of Lemma \ref{F11} and the definition of $C_2$ implying that
 $C_2\le C_{\mathcal{V}}$. \\
Therefore we deduce that 
\begin{equation}
\label{G2-est}
\mathcal{V}(t)\geq H(t), \quad \forall \ t \ge \tilde{T}_3.
\end{equation}
Now, employing the H\"{o}lder inequality and the estimates \eqref{psi} and \eqref{F2def}, we obtain
\begin{equation}
\begin{array}{rcl}
\d H'(t) &\geq&\d \frac{1}{16} \mathcal{V}^p(t)\left(\int_{|x|\leq \phi_k(t)+R}\psi(x,t)dx\right)^{-(p-1)} \vspace{.2cm}\\ &\geq& C \mathcal{V}^p(t) \rho^{-(p-1)}(t)e^{-(p-1)\phi_k(t)}(\phi_k(t))^{-\frac{(N-1)(p-1)}2}.
\end{array}
\end{equation}
In view of \eqref{est-rho}, we see that  
\begin{equation}
\d H'(t) \geq C \mathcal{V}^p(t) t^{-\frac{\left[(N-1)(1-k)+k+\mu\right](p-1)}{2}}, \ \forall \ t \ge \tilde{T}_3.
\end{equation}
From the above estimate and \eqref{G2-est}, we have
\begin{equation}
\label{inequalityfornonlinearin}
H'(t)\geq C H^p(t) t^{-\frac{\left[(N-1)(1-k)+k+\mu\right](p-1)}{2}}, \quad \forall \ t \ge \tilde{T}_3.
\end{equation}
Since $H(\tilde{T}_3)=C_2 \e>0$, 
we  easily obtain the  blow-up in finite time for the functional $H(t)$, and consequently the one for $\mathcal{V}(t)$ due to \eqref{G2-est}. 

The proof of Theorem \ref{th_u_t} is now achieved.

\section*{aknowledgments}
The authors are deeply thankful to the anonymous reviewer for the valuable remarks that improved the paper. 
A. Palmieri is supported by the Japan Society for the Promotion of Science (JSPS) – JSPS Postdoctoral Fellowship for Research in Japan (Short-term) (PE20003).


\bibliographystyle{plain}

\end{document}